\documentclass[11pt]{amsart}
\usepackage{amsmath} 
\usepackage{amsthm,amsfonts,amssymb,mathrsfs,amscd,amstext,amsbsy} 
\usepackage{epic,eepic} 
\usepackage{yfonts}
\usepackage{paralist,enumerate}
\usepackage[all]{xy}
\usepackage{hyperref}
\hypersetup{colorlinks}
\usepackage{tikz}
\usepackage{graphicx,subcaption}
\usetikzlibrary {positioning}

\newtheorem{theorem}{Theorem}[section]

\newtheorem{theo}[theorem]{Theorem}
\newtheorem{lem}[theorem]{Lemma}

\newtheorem{exa}[theorem]{Example}

\newtheorem*{Definition*}{Definition}
\def\qed{\hfill \ifhmode\unskip\nobreak\fi\quad\ifmmode\Box\else$\Box$\fi\\ }

\begin{document}

\title[Spin structures on complex projective spaces and circle actions]{Spin structures on complex projective spaces and circle actions}
\author{Donghoon Jang}
\address{Department of Mathematics, Pusan National University, Pusan, Korea}
\email{donghoonjang@pusan.ac.kr}
\thanks{This work was supported by a 2-Year Research Grant of Pusan National University.}
\begin{abstract}
It is known that the complex projective space $\mathbb{CP}^n$ admits a spin structure if and only if $n$ is odd. In this paper, we provide another proof that $\mathbb{CP}^{2m}$ does not admit a spin structure, by using a circle action.
\end{abstract}
\maketitle

\section{Introduction}

It is known that the complex projective space $\mathbb{CP}^n$ admits a spin structure if and only if $n$ is odd. This can be seen as follows. A manifold admits a spin structure if and only if its second Whitney-Stiefel class is zero \cite{Hae}. The second Whitney-Stiefel class is modulo 2 reduction of the first Chern class. The first Chern class of $\mathbb{CP}^n$ is $(n+1)x$, where $x$ is a generator of $H^2(\mathbb{CP}^n;\mathbb{Z})$. 

In this paper, we provide another proof that $\mathbb{CP}^{2m}$ does not admit a spin structure if $m>0$, by using a circle action.

\begin{theo} \label{p18}
The complex projective space $\mathbb{CP}^{2m}$ does not admit a spin structure, for any positive integer $m$.
\end{theo}

\section{Background}

Let $M$ be an orientable manifold with a Riemannian metric $g$. A \textbf{spin structure} on $M$ is an equivariant lift $P$ (called a principal $\textbf{Spin}(n)$-bundle) of the oriented orthonormal frame bundle $Q$ (called the principal $\textbf{SO}(n)$-bundle) over $M$ with respect to the double covering $\pi: \textbf{Spin}(n) \to \textbf{SO}(n)$. An orientable Riemannian manifold $(M,g)$ with a spin structure is called a \textbf{spin manifold}.

Let $M$ be a compact oriented spin manifold. The $\Hat{A}$-genus is the genus belonging to the power series $\frac{\sqrt{z}/2}{\sinh(\sqrt{z}/2)}$. Atiyah and Hirzebruch proved that if a compact oriented spin manifold $M$ admits a non-trivial smooth action of the circle group $S^1$, the equivariant index of the Dirac operator on (the spinor bundle of) $M$ is equal to the $\Hat{A}$-genus of $M$ and it vanishes.

\begin{theo} \cite{AH} \label{t23}
Let the circle act on a compact oriented spin manifold. Then $\Hat{A}(M)=0$.
\end{theo}

Let the circle act on a $2n$-dimensional oriented manifold $M$. Let $p$ be an isolated fixed point. The tangent space $T_pM$ at $p$ has a decomposition
\begin{center}
$T_pM=\bigoplus_{i=1}^n L_i$
\end{center}
into real 2-dimensional irreducibles $L_i$. For each $i$, we choose an orientation of $L_i$ so that the circle acts on $L_i$ by multiplication by $g^{w_{p,i}}$ for all $g \in S^1 \subset \mathbb{C}$, for some positive integer $w_{p,i}$. The positive integers $w_{p,i}$ are called \textbf{weights} at $p$. Let $\epsilon(p)=+1$ if the orientation on $M$ agrees with that on the representation space $\bigoplus_{i=1}^n L_i$, and $\epsilon(p)=-1$ otherwise. 
If $M$ admits a spin structure, the following formula holds.

\begin{theo} \cite{AH} \label{t25}
Let the circle act on a compact oriented spin manifold $M$ with a discrete fixed point set. Then
\begin{center}
$\displaystyle 0=\Hat{A}(M)=\sum_{p \in M^{S^1}} \epsilon(p) \cdot \prod_{i=1}^n \frac{t^{\frac{w_{p,i}}{2}}}{1-t^{w_{p,i}}}$
\end{center}
for all indeterminates $t$.
\end{theo}

A standard linear action of the circle on the complex projective space $\mathbb{CP}^n$ has $n+1$ fixed points. In particular, we consider the following action.

\begin{exa} \label{e2}
Let the circle act on $\mathbb{CP}^{n}$ by
\begin{center}
$g \cdot [z_0:\cdots:z_m]=[z_0:g z_1: g^2 z_2:\cdots:g^{n} z_{n}]$.
\end{center}
The action has $n+1$ fixed points $p_0=[1:0:\cdots:0]$, $p_1=[0:1:0:\cdots:0]$, $\cdots$, $p_n=[0:\cdots:0:1]$. At each fixed point $p_i=[0:\cdots:0:1:0:\cdots:0]$, using local complex coordinates $(\frac{z_0}{z_i},\cdots,\frac{z_{i-1}}{z_i},\frac{z_{i+1}}{z_i},\cdots,\frac{z_{n}}{z_i})$, the circle acts near $p_i$ by
\begin{center}
$g \cdot (\frac{z_0}{z_i},\cdots,\frac{z_{i-1}}{z_i},\frac{z_{i+1}}{z_i},\cdots,\frac{z_{n}}{z_i})=(\frac{z_0}{g^i z_i},\cdots,\frac{g^{i-1} z_{i-1}}{g^i z_i},\frac{g^{i+1} z_{i+1}}{g^i z_i},\cdots,\frac{g^{n} z_{n}}{g^i z_i})$

$=(g^{-i} \frac{z_0}{z_i},\cdots,g^{-1} \frac{z_{i-1}}{z_i},g \frac{z_{i+1}}{z_i},\cdots,g^{n-i} \frac{z_{n}}{z_i})$
\end{center}
for all $g \in S^1 \subset \mathbb{C}$. Therefore, as complex $S^1$-representation, the weights at $p_i$ are $\{j-i\}_{0 \leq j \leq n, j \neq i}$. Since $p_i$ has $i$ negative weights as complex $S^1$-representation, as an oriented manifold, $\epsilon(p_i)=(-1)^i$ and the weights at $p_i$ as an oriented manifold are $\{|j-i|\}_{0 \leq j \leq n, j \neq i}$.
\end{exa}

\section{Proof}

For a compact oriented manifold admitting a circle action, the following lemma provides an obstruction to the existence of a spin structure.

\begin{lem} \label{l31}
Let $M$ be a $2n$-dimensional compact oriented manifold. Suppose that $M$ admits a circle action with a non-empty discrete fixed point set with the following property: there is a fixed point $q$ such that $\sum_{i=1}^n w_{q,i} < \sum_{i=1}^n w_{p,i}$ for all fixed points $p$ with $\epsilon(p) \neq \epsilon(q)$. Then $M$ does not admit a spin structure.
\end{lem}

\begin{proof}
Assume on the contrary that $M$ admits a spin structure. By Theorem \ref{t25},
\begin{center}
$\displaystyle 0=\sum_{p \in M^{S^1}} \epsilon(p) \cdot \prod_{i=1}^n \frac{t^{\frac{w_{p,i}}{2}}}{1-t^{w_{p,i}}}=\sum_{p \in M^{S^1}} \epsilon(p) \cdot t^{\frac{\sum_{i=1}^{n} w_{p,i}}{2}} \cdot \prod_{i=1}^{n} \left( 1+t^{w_{p,i}}+t^{2w_{p,i}}+\cdots \right)$. 
\end{center}
for all indeterminates $t$.

In the equation, the fixed point $q$ contributes a term $\epsilon(q) \cdot t^{\frac{\sum_{i=1}^{n} w_{q,i}}{2}}$. Since $\sum_{i=1}^n w_{q,i} < \sum_{i=1}^n w_{p,i}$ for every fixed point $p$ with $\epsilon(p) \neq \epsilon(q)$, the term $\epsilon(q) \cdot t^{\frac{\sum_{i=1}^{n} w_{q,i}}{2}}$ cannot be canceled out, which leads to a contradiction. \end{proof}

Using Lemma \ref{l31}, we prove Theorem \ref{p18}.

\begin{proof}[\textbf{Proof of Theorem \ref{p18}}]
Consider the circle action on $\mathbb{CP}^{2m}$ in Example \ref{e2}. For each fixed point $p_i$, $\epsilon(p_i)=(-1)^i$ and the weights at $p_i$ are $\{|j-i|\}_{0 \leq j \leq 2m, j \neq i}$, $0 \leq i \leq 2m$. For any integer $k$ such that $-m \leq k \leq m$, $\sum_{i=1}^{2m} w_{p_{m+k,i}}=m(m+1)+|k|^2$. Hence, $\sum_{i=1}^{2m} w_{p_j,i}>\sum_{i=1}^{2m} w_{p_m,i}$ if $j \neq m$. Therefore, by Lemma \ref{l31}, $\mathbb{CP}^{2m}$ does not admit a spin structure.
\end{proof}

\end{document}